\documentclass[letterpaper, 10 pt, conference]{ieeeconf}  % Comment this line out if you need a4paper

\IEEEoverridecommandlockouts                              % This command is only needed if 
\usepackage{graphics} % for pdf, bitmapped graphics files
\usepackage{epsfig} % for postscript graphics files
\usepackage{mathptmx} % assumes new font selection scheme installed
\usepackage{times} % assumes new font selection scheme installed
\usepackage{amsmath} % assumes amsmath package installed
\usepackage{amssymb}  % assumes amsmath package installed
\usepackage{graphicx}

\DeclareMathAlphabet{\mathcal}{OMS}{cmsy}{m}{n}
\SetMathAlphabet{\mathcal}{bold}{OMS}{cmsy}{b}{n}

\title{\LARGE \bf
Distributed learning for optimal allocation  of synchronous and converter-based generation
}

\author{Taouba Jouini$^{1}$, Zhiyong Sun$^{2}$ 
\thanks{*This work has received funding from the European Research Council (ERC) under the European Union's Horizon 2020 research  and innovation program (grant agreement No: 834142).}%
\thanks{$^{1}$Taouba Jouini is with the Department of Automatic Control, LTH, Lund University, Sweden. $^{2}$Zhiyong Sun is with Department of Electrical Engineering,  Eindhoven University of Technology, the Netherlands.
        E-mails:
        \tt\small taouba.jouini@control.lth.se, z.sun@tue.nl.}}%

\graphicspath{{./fig/}}
\usepackage[usenames,dvipsnames]{color}
\usepackage{epstopdf}

\usepackage{amsmath,amsthm,amssymb,mathrsfs,dsfont}
\usepackage{amsfonts}
\usepackage{siunitx}
\usepackage[colorinlistoftodos,prependcaption,textsize=tiny]{todonotes}
\usepackage{tikz}
\usetikzlibrary{matrix}

\usepackage{cite}

\newcommand\oprocendsymbol{\hbox{$\blacksquare$}}

\newcommand\oprocend{\relax\ifmmode\else\unskip\hfill\fi\oprocendsymbol}

\newcommand{\real}[0]{\mathbb R}

\providecommand{\norm}[1]{\lVert#1\rVert}

\newtheorem{theorem}{Theorem}[section]

\newtheorem{definition}[theorem]{Definition}

\newtheorem{assumption}[]{Assumption}

\usepackage[normalem]{ulem}
\newcommand\rout{\bgroup\markoverwith{\textcolor{red}{/}}\ULon} % red xout

\usepackage[prependcaption,colorinlistoftodos]{todonotes}
\begin{document}
\maketitle
\thispagestyle{empty}
\pagestyle{empty}

%%%%%%%%%%%%%%%%%%%%%%%%%%%%%%%%%%%%%%%%%%%%%%%%%%%%%%%%%%%%%%%%%%%%%%%%%%%%%%%%
\begin{abstract}

Motivated by the penetration of converter-based generation into the electrical grid, we revisit the classical log-linear learning algorithm for optimal allocation {of synchronous machines and converters} for mixed power generation. The objective is to assign to each generator unit a type (either synchronous machine or DC/AC converter in closed-loop with droop control), while minimizing the steady state angle deviation relative to an optimum induced by unknown optimal configuration of synchronous and DC/AC converter-based generation. Additionally, we study the robustness of the learning algorithm against a uniform drop in the line susceptances and with respect to a well-defined feasibility region describing admissible power deviations. We show guaranteed probabilistic convergence to maximizers of the perturbed potential function with feasible power flows and demonstrate our theoretical findings via simulative examples of power network with six generation units.
\end{abstract}

%%%%%%%%%%%%%%%%%%%%%%%%%%%%%%%%%%%%%%%%%%%%%%%%%%%%%%%%%%%%%%%%%%%%%%%%%%%%%%%%

\section{INTRODUCTION}
% General intro on power systems
With increased share of renewable energy resources (wind, solar, fuel cells, etc.) in the electrical grid, it is of uttermost importance to understand the ramification of {  penetration of converter-based  generation} on the operation and maintenance of normal grid conditions. Since a large percentage of power system generation will {  withdraw} synchronous machines {  and replace these gradually by power converters} \cite{paolone2020fundamentals}. This motivates the ongoing {  efforts} to {  optimize} the grid performance in {  mixed power generation, i.e., in the presence of synchronous machines and converters} using tools from optimization theory \cite{9116977}.

% Literature review 
In particular, the field of game theory has gained more and more attention over the years as it intersects several disciplines. Game theory offers a rich set of model elements, solution concepts, and evolutionary notions. Within the realm of engineering systems, a key element in the use of game theory is to design incentives to obtain desirable behaviors \cite{marden2013distributed, bauso2016game}. 
A game theoretical model is more than just a dynamic feedback system, as each player learns the environment, which in turn learns the player \cite{bauso2016game}.

Different game theoretic concepts have been previously applied to formulate and solve optimization problems in power systems.  For example, in \cite{barreiro2018distributed} a feedback controller is proposed based on population games to achieve frequency regulation and economic efficiency. The authors of \cite{chakraborty2017distributed, 4907240} design distributed control laws for decision process of individual sources in small-scale DC power systems and propose a proportional allocation mechanism using non-cooperative game theory, respectively. 
In \cite{chakraborty2018sharing}, coalitional game approach allows consumers to cooperatively share storage with each other, while minimizing electricity consumption cost. Moreover, game theory has been leveraged for pricing and market mechanisms in power systems \cite{7105430, poolla2020market, forouzandehmehr2014stochastic} as well as security assessment and mitigation of attacks \cite{7127052, alpcan2010network}.

Learning in game theory provides a framework for designing, analyzing and controlling multi-agent systems \cite{fudenberg1998theory} and is well-understood in the literature with guaranteed asymptotic results, for example, convergence to a Nash equilibrium. Learning algorithms have been extensively studied in potential games \cite{monderer1996potential,4782013} and are designed with the goal to implement a prescriptive control approach, where the guaranteed limiting behavior  represents a desirable operating condition \cite{marden2012revisiting}. In particular, log-linear learning originally introduced in \cite{blume1993statistical} is a learning algorithm that ensures that the action profiles that maximize the global objective of the multi-agent system coincide with the potential function maximizers. The inclusion of the noise function enables the players to occasionally make mistakes corresponding to sub-optimal actions. As the noise vanishes, the probability that a player selects a best response or an optimal action goes to one \cite{marden2012revisiting}. 

% Contributions
In this work, our contributions are put together as follows: First, we consider the log-linear learning algorithm in its classical formulation \cite{blume1993statistical} through new lenses by applying it to an optimal allocation {  of synchronous and converter-based generation} for radial power systems, where the goal is to assign a type (DC/AC converter in closed-loop with droop control or synchronous machine) to each generation unit, by minimizing steady state angle deviations with respect to an optimum. {  Our study aims to understand the ramifications of the penetration of converter-based generation on the construction of future microgrids, where synchronous machines are gradually withdrawn and replaced by DC/AC converters with suitable control \cite{zhang2010new}. Our analysis helps for e.g., making the decision on which synchronous machines to withdraw in the electrical grid, while gradually replacing these with converters at different penetration levels \cite{markovic_stanojev_vrettos_aristidou_hug_2019}.}
Second, we investigate the robustness of the learning algorithm against uniform drop in the line susceptances, representing uncertainty in the knowledge of their exact value, while steady state power deviations are confined to well-defined feasibility region. This is performed by determining an upper bound on the allowed line susceptance drop. Additionally, we show that the learning algorithm for optimal allocation {  of generation units} converges, in the probabilistic sense, to an optimal configuration that corresponds to {  minimizers} of the perturbed potential function {  representing the overall cost function of the potential game}. 
Third, we validate our findings on a power system setup consisting of six generation units arranged according to a line graph. We simulate the power system with unperturbed as well as perturbed weight susceptances and discuss the convergence to an unknown optimal configuration for a susceptance perturbation within the derived theoretical bound.

% Sections presentation
The remainder of this paper unfurls as follows. Section~\ref{sec:learning} introduces the power systems model, as well as game-theoretic setup and formulates  optimal allocation problem for radial power systems. In Section \ref{sec: robustness}, we study the robustness of the learning algorithm with respect to uniform drop in the line susceptances, provide interpretation to our results and link these to other well-studied optimization problems. Finally, Section \ref{sec:sim} illustrates our results by providing numerical simulations of the learning algorithm for optimal allocation on a network consisting of aligned six- generation units.

\section{Learning for optimal allocation in power systems}
\label{sec:learning}
\subsection{Modeling of power systems}
We consider a power network model, defined by a graph ${G}=(\mathcal{V},\mathcal{E})$ in radial (or acyclic) undirected network {  (see \cite{schiffer2018almost, simpson2013synchronization})}, where $\mathcal{V}$ is the set of $n$ (possibly) heterogeneous generation units, {  i.e., synchronous and converter-based generation. These can also be encountered for example in generation networks undergoing construction, where synchronous machines are gradually withdrawn or microgrids, where the converters have heterogeneous control parameters.}
Let $\mathcal{E}$ be the set of $m$ edges (purely inductive transmission lines) with  susceptance weight $b_{e}>0, \, e\in\mathcal{E}$. The voltage magnitude $V_i$ at the $i$-th bus is assumed to be constant and equal to 1 per unit.
We consider inductive loads with constant susceptances, absorbed in the lines. Define $\{1,\dots, n\}$ as the index set of all the generation units in the power system.

We denote by $\mathcal{I}\in\real^{n\times m}$ the incidence matrix of the graph ${G}$ and by $\mathcal{N}_i$  the neighbor set of the $i$-th generation unit (synchronous machine \cite{kundur1994power}, or DC/AC converter in closed-loop with the droop control~\cite{simpson2013synchronization}). Under the assumption of quasi-stationary steady state, the swing equation of the $i$-th synchronous machine with {  known} inertial constant $m_i>0$ {  and} damping coefficient $d_i>0$ describes the $i$-th generation unit dynamics as follows,
\begin{align}
\label{eq:phase-i}
m_i \ddot\theta_i(t) &= -d_i \dot\theta_i(t)-\sum_{j\in \mathcal{N}_i}\, b_{ij}  \sin(\theta_i(t)-\theta_j(t))+P_{0,i},
\end{align}
where $\theta_i(t)\in\real$ denotes the (virtual) voltage phase angle, $P_{0,i}\in\real$ is a constant power input that represents mechanical or DC side power input and $P_{i}=\sum_{j\in \mathcal{N}_i} b_{ij}\sin(\theta_i(t)-\theta_j(t))$ is the electrical power injected from the $i$-th generation into the neighbor set $\mathcal{N}_i$. The dynamics in \eqref{eq:phase-i} describes both DC/AC converters that are equipped with droop control, or the rotor dynamics of synchronous machines, where the difference between the two models lies in the {  parameter} values of the damping $d_i$ and inertia $m_i$.

%With some abuse of notation, let $\theta_i$ denote the steady state angle of the $i$-th converter 

In the sequel, we denote the phase angles of the generation units at steady state by $\theta^s=\begin{bmatrix}\theta_1^s ,\dots, \theta_n^s\end{bmatrix}^\top$, induced by configuration (or arrangement) of the generation units $s=[s_1, \dots ,s_n]^\top$, where $s_i\in\mathcal{S}=\{M,C\}$ is the type (synchronous machine (M), or DC/AC converter (C)) of the $i$-th generation unit. The steady state angles are described by the following equation, 
\begin{align}
\label{eq: ss-eq}
d_i \dot\theta_i^s= -\sum_{j\in\mathcal{N}_i} b_{ij} \sin(\theta^s_{ij})+P_{0,i},\; \theta^s_{ij}=\theta^s_i-\theta^s_j.
\end{align} 

After defining {  the steady state power vector} $P_0=\begin{bmatrix} P_{0,1}, \dots, P_{0,n} \end{bmatrix}^\top$, {  the damping matrix} $D=\text{diag}\{d_i\}_{i=1\dots n}$, {  as well as} $\mathds{1}_n$ the vector of all ones, we make the following assumption.

\begin{assumption}[Flow feasibility \cite{simpson2013synchronization}]
\label{ass: flow-feas}
Consider the power system model at steady state in \eqref{eq: ss-eq}. We assume that
\begin{align}
\label{eq: pow-feas}
\vert \vert \mathrm{diag} (\{b_{ij}\}_{(i,j)\in\mathcal{E}})^{-1} \cdot \xi^s \vert\vert_{\infty}<1,    
\end{align}
for all configurations $s$, where $\xi^s\in\real^m$ is the vector of edge flows satisfying, 
\begin{align}
\label{eq: pow-bal}
P_{0}-\omega_0 \, D \,\mathds{1}_n= \mathcal{I}\,\xi^s,
\end{align}
with $\omega_0=\sum_{i=1}^n P_{0,i}/\sum_{i=1}^n d_i$, {  and $\norm{\cdot}_\infty$ denotes the supremum norm.}
\end{assumption}

For a given configuration $s=[s_1,\dots,s_n]^\top$ {  and} damping matrix $D$, the steady state angle {  differences} $\theta_{ij}^s$ are determined from,
\begin{align}
\label{eq: ss-angles}
\underline \sin(\mathcal{I}^\top \theta^s)=\text{diag}(\{b_{ij}\}_{(i,j)\in\mathcal{E}})^{-1} \xi^s,  
\end{align}
with $\underline \sin(z)=[\sin(z_1), \dots, \sin(z_m)]^\top$ and the edge flow vector $\xi^s$ is defined in \eqref{eq: pow-bal}.

It is important to mention that, the damping parameter of synchronous machines is considered to have higher value, than that of DC/AC converters in closed loop with droop control and enters the power balance equation \eqref{eq: pow-bal} through the damping matrix $D$. {  The values of the damping coefficients are obtained from the design characteristics of synchronous machines or converters at hand and can be tuned via torque control for machines or proportional control on the DC side for converters}. This establishes the link between a configuration $s=[s_1,\dots, s_n]$ and the {  corresponding} relative steady state angles $\theta^s_{ij}$.
%Thus, the steady state angles depend on the type of the $i$-th generation unit $s_i$, via the damping values $d_i$.
{  Physically, the parametric condition \eqref{eq: pow-feas} in Assumption \ref{ass: flow-feas} states that the active power flow along each edge be feasible, i.e., less than the physical maximum.} It is equivalent to the existence of unique and synchronized steady state angles $\theta^s=[\theta_1^s, \dots, \theta_n^s]^\top$, that are phase cohesive, i.e., $\theta^s_{ij}<\gamma, \gamma \in[0,\pi/2[$ for all $i,j\in\mathcal{V}$ satisfying \eqref{eq: ss-angles}. For more details about the derivation of \eqref{eq: pow-bal} and \eqref{eq: ss-angles}, we refer the interested reader to Theorem 2 in \cite{simpson2013synchronization}.

{  Let $  \theta_{ij}^{s^*}, (i,j)\in\mathcal{E}$  denote known optimal phase angle differences corresponding to an optimal, yet unknown configuration} $[s^{\text *}_1,\dots,s^{\text *}_n],\, s^*_i \in\mathcal{S}=\{M,C\}$. These can be obtained, for example, from historical records of the optimal power system operation {of other microgrids or power generation under similar setups}.

\subsection{Game-theoretic setup} 
\label{subsec: unper-case}
%Explain the setup: Players, Game, rules, etc.
Consider a game $(\mathcal{V},\mathcal{A},\{u_i\}_{i\in\mathcal{V}})$, where $\mathcal{V}$ is a finite set of players and $\mathcal{A}$ is a set of actions. Each player $i\in\{1,\dots,n\}$ has a utility function $u_i$ (a.k.a reward or payoff function), which associates with every action $x\in\mathcal{X}=\mathcal{A}^{\mathcal{V}}$ the utility $u_i(x)$ that player $i$ gets, when every other player $j\in\mathcal{N}_i$ is playing action $x_j\in\mathcal{A}$.
Let $x_{-i}=(x_1,x_2,\dots, x_{i-1},x_{i+1},\dots,x_n)$ denote the profile of all players' actions, other than player $i$.

\begin{definition}[Potential game, \cite{marden2013distributed}]
A game $(\mathcal{V},\mathcal{A},\{u_i\}_{i\in\mathcal{V}})$ is called a potential game, if there exists a function $U: \mathcal{X} \to \real$ (referred to as the potential function of the game), such that for any two configurations $x,y\in\mathcal{X}$ and a player $i\in\mathcal{V}$, we have that, 
$$u_i(y_i,x_{-i})-u_i(x_i,x_{-i})=U(y_i,x_{-i})-U(x_i,x_{-i}).$$
\end{definition}
A potential game as defined above requires perfect alignment between the global objective and the players’ local objective functions, so that the change in a player's utility that results from a unilateral change in strategy equals the change in the global utility \cite{marden2013distributed}.

\begin{definition}[Pure strategy Nash equilibrium]
\label{def: nash-eq}
A (pure strategy) Nash equilibrium for the game $(\mathcal{V},\mathcal{A},\{u_i\}_{i\in\mathcal{V}})$ is an action configuration $x^*\in\mathcal{X}$, such that, $$x^*_i\in\mathcal{B}_{i}(x^*_{-i}), \, i\in\mathcal{V},$$ where $\mathcal{B}_{i}(x_{-i})=\text{argmax}_{x_i\in\mathcal{A}} u_i(x_i,x_{-i})$ is the best response function.
\end{definition}

A pure Nash equilibrium as given in Definition \ref{def: nash-eq}, represents a configuration, in which no player has a unilateral incentive to deviate from his current action.

\begin{definition}[Noisy best response]
\label{def: noisy}
Consider a game $(\mathcal{V},\mathcal{A}, \{u_i\}_{i\in\mathcal{V}})$. The continuous-time asynchronous noisy best response dynamics is a Markov chain $X(t)$ with state space $\mathcal{X}=\mathcal{A}^{\mathcal{V}}$, where each player is equipped with an independent rate-1 Poisson clock. If the clock ticks at time $t$, the player $i$ updates his actions to $s$, chosen from a conditional probability:
\begin{align}
\label{eq: cdt-prob}
P(X_i(t+1)=s \vert \, X_i(t))=\frac{e^{1/\tau(t) \cdot u_i(s,X_{-i})}}{\sum\limits_{X_i\in\mathcal{A}_i} e^{1/\tau(t) \cdot u_i(X_i,X_{-i})}},
\end{align}
where $\tau(t)>0$ is the temperature or noise function that controls the smoothness of \eqref{eq: cdt-prob} and is a decreasing function of time \cite{marden2012revisiting}. 
\end{definition}
Note that $\tau(t)$ determines how likely is player $i$ to select a sub-optimal action: As $\tau(t)\to \infty$, player $i$ will select any unit type $X_i$ with equal probability. As $\tau(t)\to 0$, the only stochastically stable states (see Definition 4, in \cite{hasanbeig2018game}) of the Markov process are the joint actions that maximize the potential function. This shows the probabilistic convergence of the log-linear learning algorithm, and hence that of the learning algorithm for optimal allocation, to a Nash equilibrium that maximizes the potential function \cite{marden2009cooperative, marden2012revisiting}.
%Here again, the greater the temperature as $\tau(t)\to \infty$, the closer the conditional probability \eqref{eq: cdt-prob} {  gets} to a uniform distribution over the player action response.

The algorithm described in Definition \ref{def: noisy} is known as {\em log-linear} algorithm, and is well studied in game theory and classically described by the following rules \cite{marden2012revisiting}:
\begin{itemize}
    \item Players' utility functions constitute a potential game.
    \item Players update their strategies one at a time, which is referred to as {\em asynchrony}.
    \item At any stage, a player can select any action in the action set.
    \item Each player assesses the utility for alternative actions, assuming that the actions of all other players remain fixed.
\end{itemize}

\subsection{Optimal allocation problem} 

In this section, our goal is to assign a type $s_i\in \mathcal{S}=\{M, C\}$ to each generation unit $i\in\{1\dots n\}$, represented by a Markov chain $X(t)=[X_1(t), \dots, X_n(t)]^\top$ with $X_i(t)=s_i\in \mathcal{S}$ being the generation unit type of player $i$ at time $t$. 
This assignment is performed so that, the phase angles at each unit are the closest to their given optimal relative angles $\theta_{ij}^{s^*}$, derived from an unknown optimal configuration $s^*$ of the converters and synchronous machines.
In other words, we aim to find a valid generation unit assignment $s=[s_1,\dots,s_n]^\top$ by finding $X=[X_1,\dots X_n]^\top$, so as to minimize the pairwise interaction cost,  
\begin{align}
\label{eq: cost}
c\,(X_i(t)=s_i,\, X_j(t))=\vert \sin(\theta^s_{i{j}})- \sin(\theta^{s^*}_{ij}) \vert,
\end{align}
where $\theta^s_{ij}=\theta^s_{i}-\theta^s_{j}$ and $\theta^{s^*}_{ij}=\theta^{s^*}_{i}-\theta^{s^*}_{j}$ are steady state angle relative differences, resulting from choosing the type $s_i$ for the $i$-th generation unit, given the  generation types of the neighboring units $s_j, j\neq i$. 
 
The {  interaction} cost  in \eqref{eq: cost} depends implicitly on the configuration $X=s,\, s_i \in\mathcal{S}, \, i=1\dots n$, through \eqref{eq: pow-bal} and \eqref{eq: ss-angles}.

At every time instance $t$, one generation unit $g(t)\in\{1\dots n\}$ chosen uniformly at random wakes up and updates its type $s_i\in\mathcal{S}$. The conditional probability of updating the $i$-th generation unit to type $s_i$  is given by, 
\begin{align}
\label{eq: cond-opt-share}
P(X_i(t+1)&=s_i \; \vert \; X(t),\, g(t)=i) \\
&=\!\frac{\!\! e^{-1/\tau(t) \sum_{j\in\mathcal{N}_i} b_{ij} \,c\,(s_i,X_j(t))}}{\!\!\sum\limits_{k\in \mathcal{S}} e^{-1/\tau(t)\sum_{j\in\mathcal{N}_i} b_{ij} \,c\,(k, X_j(t))}}, \nonumber
\end{align}
where the {  temperature} function $\tau(t)$ is defined as in \eqref{eq: cdt-prob}.

We assign to each player an objective function that captures the player's marginal contribution to the potential function. This translates to assigning to each player the following utility function, 
\begin{align}
\label{eq:utility}
u_i(X_i,X_{-i})=-\sum\limits_{j\in\mathcal{N}_i} b_{ij} \,c\,(X_{i},X_j).
\end{align}

Next, we define the potential function given by, 
\begin{align}
\label{eq:pot-fcn}
U(X)=-\frac{1}{2}\sum\limits_{\substack{i,j\in\mathcal{V}, \\ (i,j)\in\mathcal{E}}} b_{ij}\, c\,(X_i(t), X_j(t)).   
\end{align}

Following Assumption \ref{ass: flow-feas}, the aforementioned potential function \eqref{eq:pot-fcn} achieves zero if and only if all generation units are aligned with an optimal configuration $s^*=[s_1^*,\dots,s_n^*]^\top$ of converters and synchronous machines. This can be seen from,
$$\vert \sin(\theta^s_{i{j}})-\sin(\theta^{s^*}_{ij})\vert =0,$$
if and only if $\theta^s_{ij}=\theta^{s^*}_{ij}$, for all $i=1,\dots, n$. In fact, given optimal steady state angle  $\theta^{s^*}$, we can write \eqref{eq: pow-bal} as $$D^*\,\mathds{1}_n=\frac{1}{\omega_0} \left[P_0-\mathcal{I}\, \text{diag} (\{b_{ij}\}) \sin(\mathcal{I}^\top \theta^{s^*}) \right]$$ and thus deduce from {  the optimal arrangement of damping coefficients $  D^*\,\mathds{1}_n=[d^*_1,\dots, d^*_n]^\top$, the  corresponding} optimal configuration $s^*$.

Note that an optimal configuration $s^*=[s_1^*,\dots s_n^*]^\top$ is a Nash equilibrium of the game, given by the utility function \eqref{eq:utility}, because it maximizes the potential function \eqref{eq:pot-fcn}. However, a Nash equilibrium $X^*=[X^*_1,\dots, X^*_n]^\top$ can be sub-optimal, i.e., $U(X^*)<0$, and hence may fail to correspond to an optimal configuration {  and thus may not be a maximizer of the potential function \eqref{eq:pot-fcn}}. 

%In a repeated potential game, the stationary distribution $p_X(t)={P}(X(t)=X)$ is given by \cite{blume1993statistical},
%\begin{align*}
%p_X(t)=\frac{e^{1/\tau(t)\cdot U(X)}}{\sum\limits_{X \in\mathcal{X}}e^{1/\tau(t)\cdot U(X)}}.
%\end{align*}

\section{Robustness of optimal allocation problem}
\label{sec: robustness}
\subsection{Drop in transmission line susceptances}
In this section, we investigate the implication of an additive unknown perturbation in transmission line susceptances on the probabilistic convergence of the log-linear algorithm and its robustness with respect to a well-defined feasibility region of admissible power flows.

Under these settings, we define the utility function for the $i$-th generation unit as follows,
\begin{align}
\label{eq: pert-u} 
\hat u_i(X)=-\sum\limits_{j\in\mathcal{N}_i} \hat b_{ij} (\delta)\, c(X_i(t), X_j(t)),
\end{align}
 where $\hat b_{ij}(\delta)= b_{ij}-\delta,\,0 \leq \delta<b_{ij}$, as well as the perturbed potential function, 
\begin{align}
\label{eq:pert-pot}
\hat U(X)=-\frac{1}{2}\sum\limits_{\substack{i,j\in\mathcal{V}, \\ (i,j)\in\mathcal{E}}} \hat b_{ij}(\delta)\, c\,(X_i(t), X_j(t)).   
\end{align}

Hence, $\delta$ represents a uniform drop in the line susceptances so that the perturbed weights are still positive \cite{ding2017impact}. A drop in the susceptance value models {  commonly encountered} uncertainty in the knowledge about the {  exact values of line parameters in power systems \cite{jones2004estimation, guo2014adaptive}}.

In the sequel, we assume that Assumption 1 holds with perturbed weights, for all $0\leq \delta< b_{ij},\, (i, j)\in\mathcal{E}$ and configurations $s$., i.e., 
\begin{align}
\label{eq: flow-fea-per}
\norm{\text{diag}(\hat b_{ij}(\delta))^{-1} \xi^{s}}_\infty <1,   
\end{align}
where $\xi^{s}$ represent edge flow vector that solves the power balance equations \eqref{eq: ss-angles} with the perturbed weights $\hat b_{ij}>0$.

Let $\theta_{ij}^s$ denote the relative steady state angles resulting from solving for 
the power balance equations \eqref{eq: pow-bal}, \eqref{eq: ss-angles} with the perturbed line susceptances $\hat b_{ij}$, for all $i,j \in\mathcal{V}$.
Next, we introduce a set of steady state angles $\theta^s$, characterizing a feasibility region described by the maximal steady state power injected by each of the converters, relative to their optimal power output, i.e., corresponding to an optimal configuration $s^*$ of the unperturbed susceptances $b_{ij}$ as follows,
\begin{align} 
\label{eq: fea-region}
&\mathcal{P}^\alpha( \theta^s)=\left\{\theta^s \in\real^n,  \right.\\
 &\left.\max\limits_{i=1\dots n}\vert \sum\limits_{j\in\mathcal{N}_i} \hat b_{ij}(\delta) \sin( \theta^s_{ij})- \sum\limits_{j\in\mathcal{N}_i} b_{ij} \sin(\theta_{ij}^{s^*})\vert< \alpha \right\},\nonumber
\end{align}
where $\sum_{j\in\mathcal{N}_i} \sin(\theta_{ij}^s)>0$ for all $s$ (by phase angle cohesiveness of $\theta_{ij}^s$). We choose $\alpha>0$, so that, $\max\limits_{i=1\dots n}\left\vert \sum\limits_{j\in\mathcal{N}_i} b_{ij} \left(\sin(\theta_{ij}^s)-\sin(\theta_{ij}^{s{^*}})\right)\right\vert <\alpha$, for all~$s$. {  Thus, for the unperturbed susceptances ($\delta=0$), the steady state angles pertain to the feasibility region \eqref{eq: fea-region}.} 
In fact, for all feasible {  power} flows (satisfying Assumption \ref{ass: flow-feas}) the region $\mathcal{P}^\alpha$ in \eqref{eq: fea-region} can also be expressed in terms of the edge flows as $\{  \xi^s \in\real^m,\, \vert \vert \mathcal{I}( \xi^{s}-\xi^{s^*})\vert \vert_\infty <\alpha\}$.
\begin{definition}
The optimal allocation robustness margin is defined by, 
\begin{align}
\label{eq: rob-marg}
 \overline{R}^s= \inf \{\delta,  \; \theta^s \not\in \mathcal{P}^\alpha( \theta^s)\}.
\end{align}
\end{definition}
This definition implies that the robustness margin is defined by the smallest susceptance drop $\delta$, that steers the steady state angles $\theta^s$, outside of their feasiblity region $\mathcal{P}^\alpha$. At this stage, we ask two fundamental questions: 
\begin{itemize}
  \item For $\tau(t) \to 0$, does the distributed learning algorithm for optimal allocation by means of the perturbed utility function \eqref{eq: pert-u}, converge in probability towards maximizers of the perturbed potential function \eqref{eq:pert-pot}?
\item Can we identify the robustness margin $\overline R^s$ and determine its dependence on network parameters/configuration?
  \end{itemize}
 
We provide answers to these questions in the following theorem.
\begin{theorem}
\label{thm: main}
{  Let $\alpha>0$ be given as in \eqref{eq: fea-region}}. Consider the power system at steady state~\eqref{eq: ss-eq} with perturbed weights $\hat b_{ij}=b_{ij}-\delta$ and the utility function \eqref{eq: pert-u}. {  Assume that \eqref{eq: flow-fea-per} holds and define,}
 \begin{align}
 \label{eq: upper-r}
    \overline R^s&=\min_{i=1,\dots, n} \frac{\alpha+\sum\limits_{j\in\mathcal{N}_i} \left(b_{ij} (\sin(\theta^s_{ij})-\sin(\theta^{s^*}_{ij}))\right)}{\sum\limits_{j\in\mathcal{N}_i} \sin( \theta_{ij}^s)}.
    \end{align} 

Additionally, consider the {  interaction} cost \eqref{eq: cost} with optimal relative angles, corresponding to unknown configuration of the power system with perturbed weights $\hat b_{ij}$. Then, the following holds:
\begin{enumerate}
  \item The distributed learning algorithm converges as $\tau(t) \to 0$, to the uniform probability over the set of best responses, which are maximizers of the perturbed potential function \eqref{eq:pert-pot}.
    \item For every configuration $s$, the power system model is robust against {  uniform} perturbations in the susceptances with a robustness margin given by $\overline R^s$ in \eqref{eq: upper-r}, if and only if $\delta< \overline R^s$.
%    \item For every optimal configuration $s^*=[s^*_1,\dots,s^*_n]^\top$, then $\theta^{s^*}\in\mathcal{P}^\alpha$.
 \end{enumerate}

\end{theorem}

\begin{proof}
Consider the cost function \eqref{eq: cost} with optimal steady state angles, that solve \eqref{eq: pow-bal} and \eqref{eq: ss-angles} with the perturbed weights $\hat b_{ij}=b_{ij}-\delta$. To prove the first statement, we follow the same lines of the proof of Proposition 3.1 from \cite{marden2012revisiting}. For this, we adopt the proof of Lemma 3.1 to our setup as follows:
The perturbed learning algorithm for optimal allocation induces a finite, irreducible, aperiodic process over the state space. By introducing $\epsilon=e^{-1/\tau}$, we denote by $P^\epsilon$ the corresponding transition matrix. 

By using the perturbed utility function in \eqref{eq: pert-u}, we have
\begin{align}
\lim_{\epsilon\to 0} \frac{P_{X\to X'}^\epsilon}{\epsilon^{\max_{X_i} u_i(X)-\hat u_i(X')}}&= \frac{1}{n} \frac{1}{\sum\limits_{X_i\in \mathcal{A}_i} \epsilon^{\max_{X_i} u_i(X)-\hat u_i(X)}}<\infty,
\end{align}
with $X=(X_i,X_{-i}),\, X'=(X'_i,X_{-i})$, if and only if, $\max\limits_{X_i} u_i(X)-\hat u_i(X)\geq 0$. This means that, for all $X_i\notin \mathcal{B}_i(X_{-i})$, we have 
\begin{align}
\frac{u_i(X)-\max\limits_{X_i} u_i(X)}{\sum\limits_{j\in\mathcal{N}_i} c(X_i,X_j)}<  \delta, 
\end{align}
and taking the {  maximum}, yields the lower bound,
\begin{align}
\underline R^s=\max_{i=1,\dots,n}\frac{u_i(X)-\max\limits_{X_i} u_i(X)}{\sum\limits_{j\in\mathcal{N}_i} c_i(X_i,X_j)}< 0.
    \end{align}
Since $0 \leq \delta$, this shows that $\underline R^s < \delta$.    
Hence, Lemma 3.1 holds with the resistance $r(X\to X^{'})=\max_{X_i} u_i(X)-\hat u_i(X')$. Together with Lemma 3.2, Proposition 3.1 in \cite{marden2012revisiting} implies that stochastically stable states are the set of perturbed potential function maximizers given in \eqref{eq:pert-pot} .

To prove the second claim, we note that $ \theta^s\notin \mathcal{P}^\alpha$ in \eqref{eq: fea-region}, if and only if, 
\begin{align*}
\vert \sum_{j\in\mathcal{N}_i} \hat b_{ij}(\delta) \sin(\theta^s_{ij})- \sum_{j\in\mathcal{N}_i} b_{ij}\sin(\theta^{\text s^*}_{ij})\vert \geq \alpha,    \end{align*}
for some $i=1\dots n$. From $\hat b_{ij}=b_{ij}-\delta$, we solve for $\delta$, the resulting inequality, 
$$ \vert \sum_{j\in\mathcal{N}_i} b_{ij}(\sin( \theta^s_{ij})-\sin(\theta^{\text s^*}_{ij})) -\delta \sum_{j\in\mathcal{N}_i} \sin( \theta^s_{ij})\vert\geq \alpha.$$
This shows that,  $$\delta\geq \frac{\alpha+\sum\limits_{j\in\mathcal{N}_i} \left[b_{ij} (\sin(\theta^s_{ij})-\sin(\theta^{s^*}_{ij}))\right]}{\sum\limits_{j\in\mathcal{N}_i}  \sin( \theta_{ij}^s)},$$ or  $$\delta\leq \frac{-\alpha+\sum\limits_{j\in\mathcal{N}_i} \left[b_{ij} (\sin(\theta^s_{ij})-\sin(\theta^{s^*}_{ij}))\right]}{\sum\limits_{j\in\mathcal{N}_i}  \sin( \theta_{ij}^s)}.$$
By choice of $\alpha$ and definition of the robustness margin in \eqref{eq: rob-marg}, only the first inequality holds and we find {  the upper bound in}  \eqref{eq: upper-r} with $\overline R^s>0$.
\end{proof}

\subsection{Discussion}
\label{subsec:disc}
\begin{enumerate}
\item Theorem 1 establishes a robustness margin $\overline R^s$ with respect to a uniform drop in the line susceptance.
The robustness margin $\overline R^s$ depends on the configuration $s$ and on admissible maximal deviation of steady state electrical power from its optimal value $  \alpha$ described in \eqref{eq: fea-region}.
Note that, optimal steady state angles of the power system with perturbed susceptances are in general unknown. In this case, the learning algorithm with perturbed susceptances does not converge to a maximizer of $\hat U(t)$. See later Section~\ref{sec:sim}.
\item A suitable choice of the temperature or noise function $\tau(t)$, as a  decreasing function of time, is crucial for the improvement of the algorithm convergence properties. The function $\tau(t)$ {  determines} the exploration rate of the learning algorithm \cite{hasanbeig2018game}.
\item The convergence in unperturbed (see Section \ref{subsec: unper-case}) as well as perturbed case (in Theorem~\ref{thm: main}) is understood only in a probabilistic sense, that is, a Markov chain $X(t)$ itself can converge to a Nash equilibrium that is {\em not a maximizer} of the potential function $U(X)$. Hence, the potential function $U(X)$ is not always necessarily zero at the end of the learning algorithm.
\item The optimal allocation {  of synchronous and converter-based generation} can be regarded in a broader perspective as a resource allocation problem~\cite{marden2013distributed}, where the cost function {  \eqref{eq:utility}} represents the welfare function. This can also be formulated as a task assignment \cite{gong2018task,qu2019distributed} for dynamic multi-agent networks, where each node selects its task, i.e., type from an admissible set and the cost function is interpreted as utility of an assignment profile for each node with a specific role.  
%An illustrative example of the distributed learning in potential games is graph coloring, whose goal is to find a color assignment, such that none of two neighboring nodes have the same color. One could interpret the type set  $\mathcal{S}=\{M,C\}$ as the colors with {  the} difference being in the cost function typically used for the coloring algorithm, see Section 6.1 in \cite{marden2013distributed}.
\end{enumerate}

\section{Simulations}
\label{sec:sim}
To illustrate the distributed learning for optimal allocation of {  synchronous and converter-based generation}, we study an example of power systems represented by a line graph of six generation units and associated with power input vector { (in p.u.)} $P_0=[0.77778, 0.7, 0.798889, 0.7,0.798889, 0.7]^\top$. The susceptances $\{b_{ij}\}_{(i,j)\in\mathcal{E}}$ are given {  (in p.u.)} by the matrix $\Gamma=\text{diag}\left\{[15.2631,4.2350, 4.8156,15.2631,4.2350]\right\}$. The values are taken from the IEEE 14-bus system \cite{4075293}.

The goal is to assign a generation type $s_i\in\mathcal{S}$ to each generation unit $i\in\{1,\dots,6\}$. The optimal yet unknown configuration $s^*$, shown in Figure~\ref{fig:setup}, {  is given by} optimal steady state angles, which are minimizers of the potential function $U(t)$ in \eqref{eq:pot-fcn}. The optimal relative angle differences $\{\theta^{s^*}_{ij}\}_{i,j\in\mathcal{V}}$ are given by $[-0.0157, -0.0354,  0.0081, {  0.0084} , 0.0750]^\top$. 
The distributed learning algorithm for optimal allocation aims to correctly assign a type (synchronous machine or DC/AC converter in closed-loop with droop control) and thus decide on which synchronous machine to withdraw and replace with DC/AC converter. The optimal configuration $s^*$, based on the given optimal angle differences $\{\theta^{s^*}_{ij}\}_{i,j\in\mathcal{V}}$.

\begin{figure}[h!]
    \centering
    \includegraphics[scale=0.2, trim=0cm 9cm 0cm 7cm, clip=true]{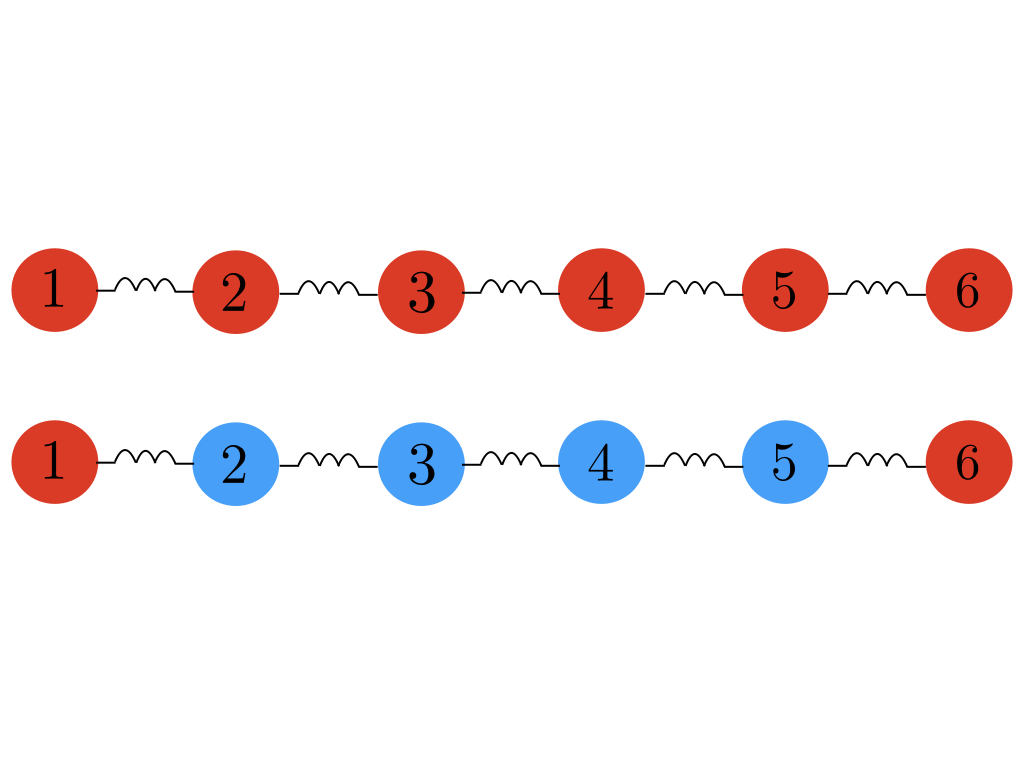}
    \caption{{\em Top}: Power system represented by a line graph consisting of six synchronous machines connected via inductive lines with susceptances $b_{ij}>~0$. {\em Bottom}: Unknown optimal configuration $s^*=[{M, C,C, C, C,M}]^\top$ corresponding to given optimal relative angles $\{\theta^*_{ij}\}_{i,j\in\mathcal{V}}$. The color {\em red} refers to a synchronous machine and {\em blue} to DC/AC converters (in closed-loop with droop control).}
    \label{fig:setup}
\end{figure}

{  For this}, we choose between two colors: {\em red} for synchronous machines $(M)$ and {\em blue} for DC/AC converters $(C)$ with respective damping values (in p.u.) $d_C=15$ and $d_M=25$. At initialization, all generation units are assumed to be synchronous machines, so that $X_i(t)=M, i=1,\dots,n$. We begin with the unperturbed ($\delta=0$) cost function~\eqref{eq:utility}. At each time $t>0$, a generation unit $i\in U(t)$ is chosen uniformly (at random) and updates its generation type $s_i$ (and hence color), according to the conditional probability \eqref{eq: cond-opt-share}. In Figure \ref{fig:potential-fcn}, we plot the time evolution of the potential function in \eqref{eq:pot-fcn}, for two different realizations of the inverse of the temperature function $\eta(t)=1/\tau(t)$. Note that an increase in the slope of $\eta(t)$ is accompanied by an increase in the convergence rate to an optimal configuration. 
\begin{figure}[h!]
    \centering
    \includegraphics[scale=0.4]{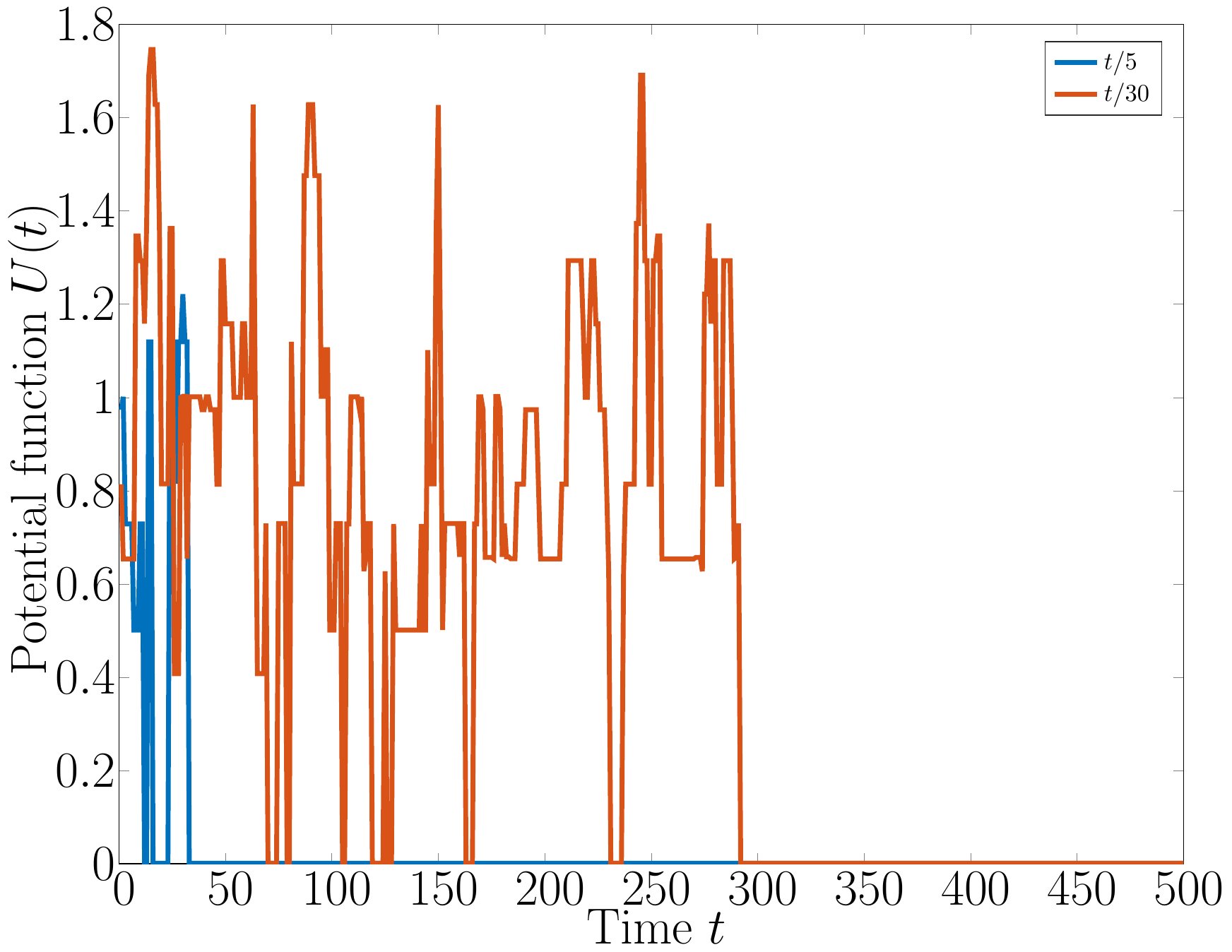}
    \caption{Evolution of the {  unperturbed} potential function $U(t)$ in \eqref{eq:pot-fcn} corresponding to two different realizations of the inverse of the temperature function $\eta_1(t)=t/5$ and $\eta_2(t)=t/30$.} 
    \label{fig:potential-fcn}
\end{figure}
%\begin{figure}[h!]
% \label{fig:opt-config}
%    \centering
 %   \includegraphics[scale=0.21,trim=0cm 10cm 0cm 9cm, clip=true]{optimal-config.jpg}
 %   \caption{Convergence to the optimal configuration $s^*$ for the distributed learning algorithm for unperturbed ($\delta=0$) and perturbed ($\delta=0.0027$) weight susceptances, corresponding to }
%\end{figure}

For $\eta(t)=t/5$, we consider the {  perturbed} distributed learning algorithm. We calculate $$\max_{\substack{s_i\in\mathcal{S}\\ i=1\dots n}}\max\limits_{i=1\dots n}\vert \sum\limits_{j\in\mathcal{N}_i} b_{ij}( \sin(\theta^s_{ij})-\sin(\theta_{ij}^{s^*})) \vert=0.8142,$$ and hence pick $\alpha=1.5>0.8142$, so that $\theta^s\in \mathcal{P}^\alpha$ in \eqref{eq: fea-region}, for all configurations $s$. This choice accounts for a large margin of admissible power deviations induced by bounded disturbances that might affect power system operation \cite{kundur1994power}.

Next, we perturb the susceptance values given by $\hat b_{ij}=b_{ij}-\delta$, where $0\leq\delta<b_{ij}$ is a uniformly randomly generated perturbation. The upper bound on the susceptance drop $\delta$ from Theorem \ref{thm: main} is given by $\min\limits_{\substack{s, s_i\in\mathcal{S}\\ i=1,\dots, n}}\overline R^s= 0.4723$. In this case, the resulting steady state angles remain inside the feasibility region $\mathcal{P}^\alpha$.
%, where for example for $\delta=0.0027$, we have that, 
%\begin{align*}&\max_{\substack{s, s_i\in\mathcal{S}\\ i=1,\dots,n}} \max\limits_{i=1\dots n}\vert \sum\limits_{j\in\mathcal{N}_i} \hat b_{ij}(\delta) \sin(\theta^s_{ij})- \sum\limits_{j\in\mathcal{N}_i} b_{ij} \sin(\theta_{ij}^{s^*})\vert\\
%&= 0.8862.\end{align*} 
Since the optimal configuration $s^*=[s_1^*,\dots, s_n^*]$ in Figure~\ref{fig:setup}, corresponds to steady state angles \eqref{eq: pow-bal}, \eqref{eq: ss-angles} with unperturbed line susceptances $b_{ij}$, Figure \ref{fig:pert-pot} shows the perturbed potential function \eqref{eq:pert-pot} (for $\delta=0.3065$) that does not converge to zero, {  that is} a Nash equilibrium that is not a minimizer of the perturbed potential in \eqref{eq:pert-pot}.

\begin{figure}[h!]
    \centering
    \includegraphics[scale=0.4]{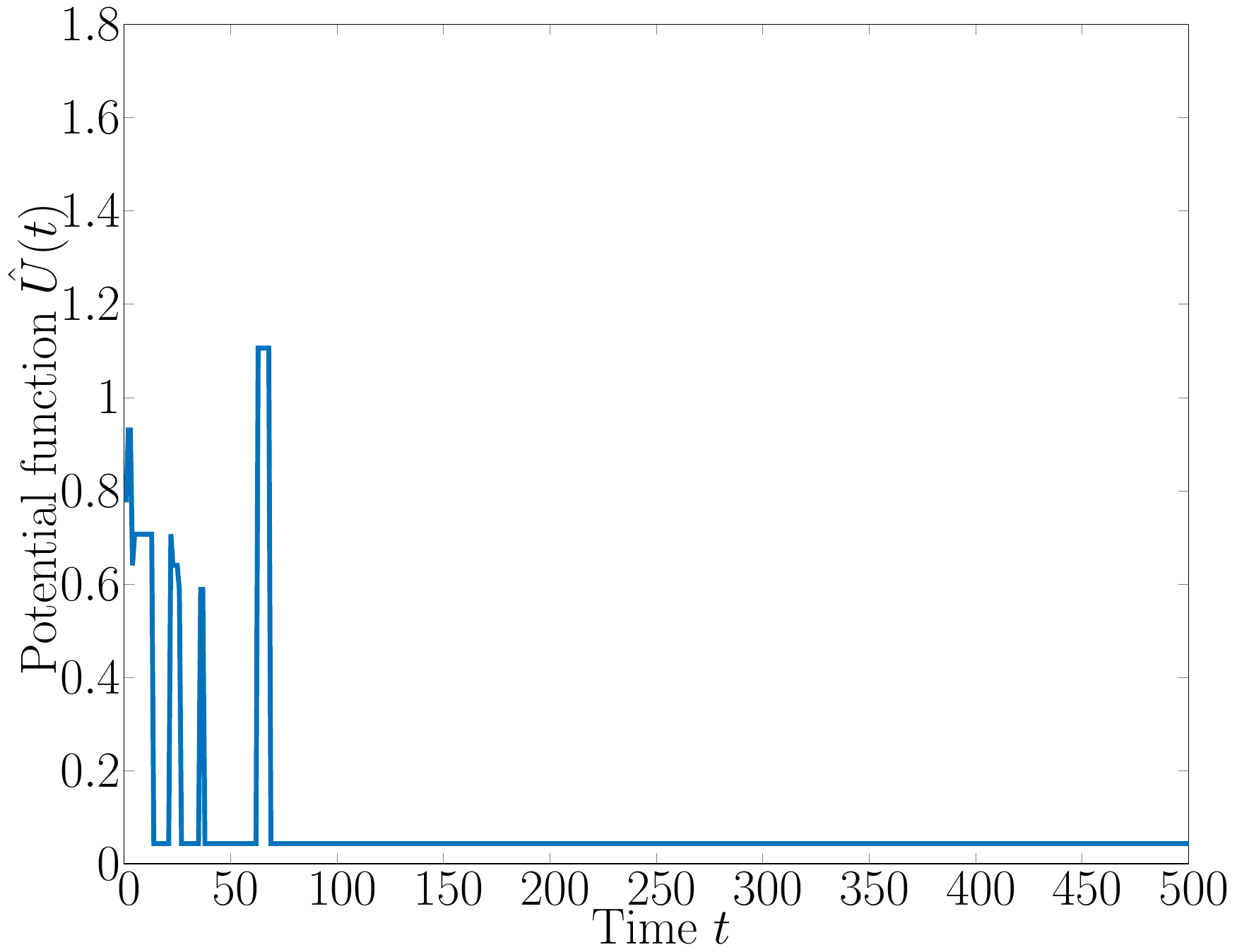}
    \caption{Convergence of the perturbed potential function $\hat U(t)$ in \eqref{eq:pert-pot} for $\delta=0.3065$ to a non-zero value corresponding to a configuration of a Nash equilibrium, that is {not} a {  minimizer} of the perturbed potential function.}
    \label{fig:pert-pot}
\end{figure}

\section{CONCLUSION}
We revisited a distributed learning algorithm for optimal allocation {  of synchronous and converter-based} generators in radial power systems based on log-linear learning with guaranteed probabilistic convergence to Nash equilibrium. Moreover, we investigated its robustness against drops in the line susceptances with respect to {  a} feasible region of power deviations. Future investigations involve the {  study of general network topology for mixed power generation.}


%%%%%%%%%%%%%%%%%%%%%%%%%%%%%%%%%%%%%%%%%%%%%%%%%%%%%%%%%%%%%%%%%%%%%%%%%%%%%%%%

%%%%%%%%%%%%%%%%%%%%%%%%%%%%%%%%%%%%%%%%%%%%%%%%%%%%%%%%%%%%%%%%%%%%%%%%%%%%%%%%

%%%%%%%%%%%%%%%%%%%%%%%%%%%%%%%%%%%%%%%%%%%%%%%%%%%%%%%%%%%%%%%%%%%%%%%%%%%%%%%%

\section*{ACKNOWLEDGMENT}
The authors thank Dr. Emma Tegling for the insightful comments and discussions.

\bibliographystyle{IEEEtran}
\bibliography{root.bib}

\end{document}